\documentclass[a4paper]{article}
\usepackage[latin1]{inputenc}
\usepackage{ae}
\usepackage{amsmath,amsthm}
\usepackage{graphicx}
\usepackage{amsfonts}
\usepackage{esvect}
\usepackage{calligra}
\usepackage{fancyhdr}
\usepackage{vector}
\usepackage[T1]{fontenc}
\usepackage{setspace}
\usepackage{cmbright}
\usepackage{amsthm}
\usepackage{calrsfs}
\usepackage[a4paper , lmargin = {3cm} , rmargin = {3cm} , tmargin = {2.5cm} , bmargin = {2.5cm} ]{geometry}
\usepackage{amssymb}
\title{Diplomarbeit}
\author{Karim Adiprasito}
\DeclareMathAlphabet{\mathpzc}{OT1}{pzc}{m}{it}
\newenvironment{frcseries}{\fontfamily{frc}\selectfont}{}
\newtheorem*{theorem*}{Theorem}
\newtheorem{theorem}{Theorem}
\newtheorem{Cor}[theorem]{Corollary}
\newtheorem{Lemma}[theorem]{Lemma}
\newtheorem{Prop}[theorem]{Proposition}

\newcommand{\conv}[1]{\textsc{conv}(#1) }
\newcommand{\e}{\varepsilon}
\newcommand{\R}{\mathbb{R}}
\newcommand{\E}{\mathcal{E}}

\newcommand{\bd}[1]{\textsc{bd}(#1) }

\newcommand{\intx}[1]{\textsc{int}(#1)}

\renewcommand{\epsilon}{\varepsilon }

\theoremstyle{remark}
\theoremstyle{definition}

\begin{document}
\title{Characterization of polytopes via tilings with similar pieces}
\author{Karim ADIPRASITO\footnote{The final preparation of this paper was supported by the Deutsche Forschungsgemeinschaft within the research training group `Methods for Discrete Structures' (GRK1408).}\\ \small Institute of Mathematics, Discrete Geometry Group, Free University Berlin, Arnimallee 2, 14195 Berlin\\
\small e-mail: karim.adiprasito@fu-berlin.de}
\date{\today}
\maketitle
\bfseries
\mdseries
\begin{abstract} Generalizing results by Valette, Zamfirescu and Laczkovich, we will prove that a convex body $K$ is a polytope if there are sufficiently many tilings which contain a tile similar to $K$. Furthermore, we give an example that this can not be improved.
\end{abstract}
Consider a convex body (a compact convex set with nonempty interior) $K$ in $\R^n$, which is tiled into a finite number of convex bodies. A tiling $T$, as a formal object, will here be the set consisting of all its tiles, and all tiles will be assumed to be convex bodies. Also, this investigation will only consider proper tilings, that is, tilings which are not trivial (meaning they consist of more than just one tile).\\
We consider the following problem: Let there be tiles similar to $K$. Does it follow that $K$ is a polytope?\\
For dimension 2, M. Laczkovich (\cite{decolac}) could show that if one tile is similar to $K$, and if the tiling is proper, $K$ is in fact a polytope. He generalized a remark by G. Valette and T. Zamfirescu in \cite{decovz}. \\
Now the question arises: is this extendable to dimension 3? Already the original paper by Laczkovich contained a remark that the immediate generalization must be wrong and left the problem open for higher dimensions.\\
Indeed, a circular cone can easily be tiled in such a way that one tile is similar to the cone by just cutting near the apex. T. Zamfirescu asked whether this was optimal, i.e.: Consider a convex body $K$ in 3-space which is tiled in such a way such that 2 tiles are similar to $K$. Then is $K$ a polytope?\\
This will turn out to be true, and is a very special case of the general theorem. In dimensions higher than 3, a condition only on the number of similar tiles will never be sufficient (see the example at the end of this paper). The additional condition will encode some information on how the tiles are located relative to the convex body itself. In consistency with our observations, the condition will degenerate in dimensions 2 and 3. We thus are able to answer Zamfirescu's question as well as solve the problem.\\
Let us state the theorem:
\begin{theorem}
\label{main}
Let $K$ be a $n$-dimensional convex body, and let $T_i,\ \{1,2,3,...,n-1\}$ be $n-1$ proper tilings of $K$, each of which contains a tile $L_i$ similar to $K$.\\
If the convex hull of the fixed points $x_{L_i}$ forms a nondegenerate $n-2$-dimensional simplex, then $K$ is a polytope.
\end{theorem}
Here, the point $x_L$ is the fixed point of the similarity $f_L$ from $K$ to $L$. If there is more than one similarity from $K$ to $L$, choose one.\\
The simplex referred to in theorem \ref{main} will be called the {\it tip simplex} of $K$ with respect to the tilings involved.\\
Often, one is in the situation to have a single tiling, and several tiles are similar to $K$. This is a special case of above theorem, which is a bit harder to prove, since then one has to deal with interdependencies between the tiles similar to $K$ when deforming tilings.
\begin{Cor}
Let $K$ be a $n$-dimensional convex body, which is properly tiled into a finite number of convex bodies, $n-1$ of which are similar to $K$. Denote these particular tiles by $L_i$.\\
If the convex hull of the fixed points $x_{L_i}$ forms a nondegenerate $(n-2)$-dimensional simplex, then $K$ is a polytope.
\end{Cor}
\begin{Cor}
Let $K$ be a convex body in $\R^n,\ n\leq3$, which is properly tiled into a finite number of convex bodies, $n-1$ of which are similar to $K$. Then $K$ is a polygon/ a polyhedron.
\end{Cor}
Even more, in dimension 3, symmetry also plays a role in this calculation:
\begin{Cor}
Let $K$ be a convex body in $\R^3$, which is properly tiled into a finite number of convex bodies, $1$ of which is similar to $K$. If there are two similarities $f_L,f_L'$ from $K$ to the tile $L$ with different fixpoints respectively, then $K$ is a polytope.
\end{Cor}
\begin{proof}
We can regard this situation in the following way: $K$ has two tilings $T$ and $T'$ which are the same, but in $T$ $f_L$ is the similarity from $K$ to $L$, and in $T'$ $f_L'$ is considered to be the similarity from $K$ to $L$. Theorem \ref{main} then concludes the proof.
\end{proof}
Denote by $\bd{M},\ \intx{M},\ \E(M),\ \conv{M}$ the the boundary/ the interior/ set of extremal points/ the convex hull of a set $M$. Also, $B_\e(x)$ denotes the set of all points in $\R^n$ with distance less than $\e$ to $x$.\\
Consider the general situation of a convex body $K$ tiled into convex bodies $K\supset P_j\in T,\ j\in [s]:=\{1,2,3,...s\}$. For every 2 tiles $P_i,\ P_j$ there is a hyperplane $H_{ij}$ seperating the two. For a fixed $i\in [s]$, let $F_{ij}$ be the halfspace with boundary $H_{ij}$ containing $P_i$. Then, $P_i=K\cap \bigcap_{i\neq j} F_{ij}$, and every extremal point $x\in \intx{K}$ of $P_i$ is a vertex of $\bigcap_{i\neq j} F_{ij}$. In particular, $\E(T):=\cup{\E(P_j)}$ has finitely many elements in common with $\intx{K}$, i.e. $\E(T)\cap\intx{K}$ is finite.
Suppose $L=P_1$ is similar to $K$, then the similarity transformation $f_L$ from $K$ to $L$ maps $\E(K)$ to $\E(L)$ bijectively. If the fixed point $x_L$ of $f_L$ lies in $\intx{K}$, then for some iteration of $f_L$, $f_L^n(K)$ lies in the interior of $K$. Since $\E(f_L^n(K))=\E(f_L^n(K))\cap \intx{K}$ is finite, so is $\E(K)$, and $K$ is a polytope. Therefore, we can assume that $x_L\in\bd{K}$. 
\section*{Adapting a tiling}
In this section, we focus on adapting tilings until they have features that come handy in the later proof. First, we want $L$ to be homothetic to $K$. Then we want to move the fixed point to an arbitrary point of the tip simplex.\\
\\
Before we begin, let us state our methods to deform a tiling:\\
Let $K$ be a convex body, and let $T$ be a tiling of $K$, containing a tile $L$ similar to $K$. Let $f_L$ be a similarity mapping from $K$ to $L$. Define $f_L(T)$ to be the function $f_L$ applied to all the tiles of $T$, and thus a tiling of $L$. The idea is that we can refine the tiling $T$ to $T'$ by defining $T':= f_L(T)\cup T \setminus \{L\}$. In this situation, we will also write $T'=f_L(T)+ T$. The similarity mappings of newly created tiles will be defined by the composition of the similarities involved. We call this procedure {\it iterating a tiling}.\\
In a related situation, suppose we are given two tilings $T_1$ and $T_2$, we could form the tiling $T_1 * T_2:=\{P\cap Q|\ P\in T_1,\ Q\in T_2\}$. If $T_1$ tiles a set $K_1$, and $T_2$ tiles $K_2$, then $T_1 * T_2$ is a tiling of the intersection of $K_1$ and $K_2$. In particular, if $K_1=K_2$, then $T_1 * T_2$ is a (possibly refined) tiling of $K_1$. Let us apply the above methods to concrete tilings to show that simplifications are indeed possible.
\begin{Lemma}
\label{homot}
Let $K$ be a convex body as above, $T$ a tiling which contains a tile $L$ similar to $K$. Then there is a tiling $T'$ of $K$ which contains a homothetic copy of $K$ whose fixed point coincides with $x_L$. 
\end{Lemma}
\begin{proof}
Denote by $f_L$ the similarity transformation from $K$ to $L$. We may assume $x_L=0$.
Let $O_n$ denote the set of orthogonal transformations of $R^n$. Let $\lambda$ denote the similarity
ratio of $f_L$, and define $M_L:=\lambda^{-1} f_L$. Then $M_L \in O_n$. Since $O_n$
forms a compact subset of the space of n-dimensional matrices, it follows that the
powers of $M_L$ can get as close to the identity matrix as we please.\\
Write $B_\e:=B_\e(0)$ for the open ball of radius $\e$ around $0$ and $TC(K):=TC_0 (K)=TC_{x_L}(K)$ for the solid tangent cone of $K$ at 0. 
Obviously, the solid tangent cones of later iterations of $f_L(K)$ are included in earlier:
$$TC(K)\supset TC(L) \supset TC(f^2_L(K)) \supset TC(f^3_L(K)) ...$$
But since we can get $f^i_L$ as close to a homothety as we want, we can get $B_\e\cap TC(f^i_L(K))$
as close to $B_\e\cap TC(K)$ (with respect to, for example, the Pompeiu-Hausdorff metric and a fixed $\e>0$) as we want by choosing $i\in \mathbb{N}$. Thus, the solid tangent cones of $K$ and $f^i_L(K),\ i\in \mathbb{N}$ at $x_L$ coincide. In particular, because all tiles are compact, convex and have nonempty interior, their tangent cones are never degenerate, and it follows that $L$ is the only tile of $T$ which contains $0=x_L$. Then there is an $\e>0$ such that
$$K\supset B_\e\cap K=B_\e\cap L=B_\e\cap TC(K).$$
In particular, $x_L$ is not an accumulation point of extremal points of $K$.\\
Choose $R$ such that $B_R\supset K$, set $t:=R/\e$ and $d$ such that $\lambda^d<\e$ and $t\lambda^d<1$. There is a tiling $T$ of $K$ such that $f_L^d(K)$ is a tile of $K$. $tM_L^{-d}(T)$ gives a tiling of $tM_L^{-d}(K)$, and by the choice of $t$, $tM_L^{-d}(K)\supset K$. Thus, $tM_L^{-d}(T)* \{K\}$ is a tiling of $K$ which contains $L':=tM_L^{-d}(f_L^d(K))=t\lambda^d K\subsetneq K$, which is a homothetic copy of $K$.
\end{proof}
Before we go on, let us consider the above constructed tiling. Recall the notion of hyperplanes $H_{1j}$ seperating the tiles $P_1=L'$ and $P_j$ of $K$, then $L'$ is the convex hull of 
$$\{x_L\}\cup (H_{1j}\cap TC(K))=\{x_L\}\cup (H_{1j}\cap K).$$ Thus, we have the following Lemma:
\begin{Lemma}
$K$ can be written as $$\conv{\{x_L\}\bigcup_i B_i},$$ where the $B_i\subset\bd{K}$ are finitely many convex compact $n-1$-dimensional sets with boundary (which can be chosen to be disjoint from $\{x_L\}$). We will call these sets a base of $K$.
\end{Lemma} 
Now that the tiles are directed nicely, we want to turn to adjusting their position.
\begin{Lemma}
Let $K$ be a convex body as above, $T_i$ tilings fulfilling the conditions of theorem \ref{main}, $S$ the tip simplex, $x_0$ a point in the tip simplex. Then there is a tiling $T'$ with a tile $L$ similar to $K$ such that $x_L=x_0$. 
\end{Lemma}
\begin{proof}
If the tip simplex is just a point, the Lemma is trivial. Suppose therefore that there are at least two tilings with $T_1$ and $T_2$ with tiles $L_1$ respectively $L_2$ which are similar to $K$ and have distinct fixed points, and suppose the corresponding similarities are chosen to be homotheties. Then the tiling $T_1+f_{L_1} (T_2)$ will contain a homothetic copy of $K$: $L_3=f_{L_1}(L_2)$. $L_3$ is new to us, because the corresponding fixed point will not coincide with $x_{L_1}$ or $x_{L_2}$, but will lie in their convex hull. By iterating this procedure, we see that there is a dense subset $M$ of $\conv{\{x_{L_1}, x_{L_2}\}}$ we can make the fixed points of similarities lie in.\\
Next, suppose $x_0\in \conv{\{x_{L_1}, x_{L_2}\}}\setminus M$. Find a tiling $T$ with a similar copy 
$L$ of $K$ such that $x_L\in \conv{\{x_{L_1},x_0\}}$. (Again, we assume the similarities to be homotheties.) Define $H(x)$ to be the homothety which fixes $x_{L_1}$ and maps $f_{L}(x_0)$ to $x_0$, and let $T':=H(T)$, which covers $K$. $T'*\{K\}$ is a tiling of $K$. If $L$ is chosen small enough, $H(L)$ is a proper subset of $K$ and thus a tile of $T'*\{K\}$ with fixed point $x_0$. With higher dimensional tip simplices, this construction works just in the same way.
\end{proof}
\section*{Conclusion of proof}
We will prove theorem \ref{main} by induction. In dimension 1, it is trivial, in dimension 2, it was proven by Laczkovich.\\
\\
Let us assume theorem \ref{main} is proven in dimension $n-1$. Let $K$ be a convex body fulfilling the conditions of the theorem \ref{main} in Dimension $n$, let $x$ be some point in the relative interior of the tip simplex $S$, and let $H$ be some $n-1$-dimensional hyperplane of $\R^n$ containing $x$. $H\cap K$ is a $n-1$-dimensional convex set. It could even be of smaller dimensions, so let us just assume $H\cap K$ is not a point. Let $T'_i$ be proper tilings of $K$ whose similarities are homotheties and whose fixed points are also the extremal points of the $n-3$-dimensional simplex $S\cap H$.\\
$H\cap K$ inherits a tiling structure from $K$ by means of intersection: $T'_i$ induces the tiling $\{H\} * T'_i$ on $H\cap K$. This tiling is proper if $H$ does not coincide with any of the hyperplanes separating an $L_i$ from an adjacent tile, which is true for all but finitely many choices of $H$. Note that because $L'_i$ is a homothetic copy of $K$ whose fixed point lies in $H$, $H\cap L'_i$ will be an element of $T'_i$ which is a homothetic copy of $H\cap K$. The fixed point of $H\cap L'_i$ in $H\cap K$ coincides with the fixed point of $L'_i$ in $K$ and is therefore an extremal point of $S\cap H$. Since $x$ is a relative interior point of $S$, $S\cap H$ is a nondegenerate $n-3$-dimensional simplex spanned by the fixed points of $H\cap L'_i$, which in turn are elements of the proper tilings $\{H\} * T'_i$ of $H\cap K$.\\
Using the induction hypothesis, we see that $H\cap K$ must be a polytope.
\begin{Prop}
\label{almost}
Let $K$ be a convex body in $\R^n$, which is tiled as in the description of theorem \ref{main}. Let $S$ be the tip simplex, $x$ a point in the relative interior of this simplex, and let $H$ be some hyperplane of $\R^n$ containing $x$. Then $H\cap K$ is a polytope in all but possibly a finite number of cases.
\end{Prop}
\begin{proof}[Proof of theorem \ref{main}]
Assume theorem \ref{main} is proven in dimension $n-1$.\\
As already stated, $K$ is the convex hull of the union of a finite number of $n-1$ dimensional compact convex sets and a point $x$ in the relative interior of the tip simplex. Since we supposed $K$ is not a polytope, one of these sets has infinitely many extremal points which are also extremal points of $K$. Call this set $B$, and recall that $B\subset \bd{K}$.
Pick another point $y$ in the relative interior of the tip simplex of $K$. Choose a $n-1$-dimensional affine subspace $H$ parallel to $B$ and containing $y$. If $H\cap \conv{\{x\}\cup B}=\emptyset$, interchange the roles of $y$ and $x$. If $H$ contains $x$, tilt $H$ just a little such that it intersects $\conv{\{x\}\cup B}$ but neither intersects $\{x\}$ nor the base $B$. Also, assume $H$ is not one of the finite hyperplanes excluded in proposition \ref{almost}.\\
The intersection of $H$ and $K$ is, as we know from proposition \ref{almost}, a polytope. But $H\cap B$ will have infinitely many extremal points, as it is a (possibly dilated, if we tilted $H$) homothetic copy of $B$. Thus, $K$ and $B$ can only share a finite number of extremal points, in contradiction with the assumption.
\end{proof}
\subsection*{Sharpness of results}
We could now ask if the above results are optimal, and indeed, they are. First note that as soon as we have constructed (in any dimension) a tiling $T$ of a convex body $K$ which contains (at least) 2 similar copies $L,L'$ of $K$, then we could create a tiling with more similar copies by regarding the tiling $f_L(T)+T$. Note however that we can never make a degenerate tip simplex nondegenerate using this method. The dimension of the tip simplex is thus the real condition to make $K$ a polytope, a property not visible in dimensions 2 or 3.\\
Using induction on dimensions, we can construct a convex body which even allows us to see that the condition on the tip simplex is optimal, in particular, $n-3$-dimensional tip simplices are not enough to conclude that $K$ is a polytope. Figuratively speaking, we take a circular cone, which shows theorem \ref{main} to be sharp in dimension 3, and take it as a base for a 4-dimensional cone, which in turn forms a base of a five-dimensional cone etc. This example is just an extension of the example known for dimension 3.\\
To make a concrete example with $n-3$ dimensional tip simplex in dimension $n>2$, use the following construction: Regard the convex body $K$ which is the set of all points
$$\sqrt{x_1^2+x_2^2}+\sum_{i\in\{3,4,5,...,n\}} x_i\leq 1;\ x_i\geq 0\ \forall i\in\{3,4,5,...,n\}$$
where the $x_i$ are coordinates with respect to some base $\{e_1,e_2,e_3,...,e_n\}$ of $\R^n$. We will show that the tip simplex $S$ is $\conv{\bigcup_{i\in\{3,4,5,...,n\}} \{e_i\}}$.\\
We regard the homotheties
$$f_i(x)=\frac{1}{2}(x-e_i) + e_i$$
for $i\in\{3,4,5,...,n\}$. Their fixed points span the said tip simplex, and the interior of the images of $K$ does not intersect. Thus, it remains to show that the remaining tile is convex. But this is simple, since it can be written as intersection of the convex sets $K$ and $X_i,\ i\in\{3,4,5,...,n\}$, 
$$X_i:=\{x\in \R^n|x_i \leq \frac{1}{2} \}.$$

\end{document}